\documentclass[11pt,reqno]{article}
\usepackage{amsmath,amssymb,amsthm}
\usepackage{graphicx}
\pdfoutput=1
\usepackage[english]{babel}
\usepackage{hyperref}
\usepackage[T1]{fontenc}
\newcommand{\nn}{\mathbb N}
\newcommand{\ee}{\mathbb E}
\newcommand{\pr}{\mathbb P}
\newcommand*{\bbrac}[1]{\bigl(#1\bigr)}
\newcommand*{\prb}[1]{\pr\bbrac{#1}}
\newcommand{\eps}{\varepsilon}
\newcommand*{\hier}[3][X]{#1^{(#2)}_{#3}}
\newcommand{\salj}{\mathcal{F}}
\newcommand{\Var}{\operatorname{Var}}
\newcommand*{\abs}[1]{\lvert#1\rvert}
\newcommand*{\babs}[1]{\bigl|#1\bigr|}
\newcommand*{\floor}[1]{\lfloor#1\rfloor}
\newcommand{\ie}{i.e.\ }

\theoremstyle{plain}
\newtheorem{theorem}{Theorem}
\newtheorem{lemma}[theorem]{Lemma}
\newtheorem{proposition}[theorem]{Proposition}
\theoremstyle{remark}
\newtheorem*{remark}{Remark}

\title{Non-convergence of proportions of types in a preferential attachment graph with three co-existing types}
\author{John Haslegrave\thanks{Research supported by the European Research Council (ERC) under
the European Union's Horizon 2020 research and innovation programme (grant agreement No.\ 639046).} \and Jonathan Jordan}

\begin{document}
\maketitle
{\centering\textit{Dedicated to the memory of Chris Cannings}\par}
\begin{abstract}We consider the preferential attachment model with multiple vertex types introduced by Antunovi\'{c}, Mossel and R\'{a}cz.
We give an example with three types, based on the game of rock-paper-scissors, where the proportions of vertices of the different types
almost surely do not converge to a limit, giving a counterexample to a conjecture of Antunovi\'{c}, Mossel and R\'{a}cz.
We also consider another family of examples where we show that the conjecture does hold.\end{abstract}


\section{Introduction}

We consider a model for randomly growing networks of agents having different types, which are not innate but are chosen based on what they see when they join the network.
These types might represent social groups, opinions, or survival strategies of organisms. This model was introduced by Antunovi\'{c}, Mossel and R\'{a}cz \cite{AMR},
who define a family of preferential attachment random graphs where each new vertex receives one of a fixed number of types according to a probability distribution
which depends on the types of its neighbours. Using stochastic approximation methods, they fully deal with the case where there are two types of vertices, and show that the
proportions of the vertices which are of each type almost surely converge to a (possibly random) limit which is a stable fixed point of a certain one-dimensional differential equation.

The case where there are more than two types is discussed in Section 3 of \cite{AMR}.  They conjecture (Conjecture 3.2) that the behaviour is always similar to the two-type
model in that there is almost sure convergence to a limit which is a stationary point of what is now a multi-dimensional vector field.  They confirm this for the case they
call the linear model, where the probability each type is chosen is proportional to the number of neighbours of that type the new vertex has.  However, the difficulties of
a general analysis of the class of vector fields associated to models of this type make it hard for more general results to be proved.

In this note, we give a model with three types which is a counterexample to Conjecture 3.2 of \cite{AMR}.  The type assignment mechanism in this model is inspired by the
well-known game of ``rock-paper-scissors'', and we show that in this model the associated vector field does not have attractive stationary points and that the proportions
of the types do not converge to a limit, but approach a limit cycle of the associated vector field, so that each type in turn will have the largest proportion of the edges.

We also give a family of examples where we can show that Conjecture 3.2 of \cite{AMR} does hold.  This is where the new vertex chooses uniformly at random from those types
represented among vertices which it connects to.  In this case, we will show that there is a single stable stationary point of the vector field, which corresponds to
the proportions of each type being equal, and that almost surely the proportions of the types will converge to this point as $n\to\infty$.

\section{The Antunovi\'{c}-Mossel-R\'{a}cz framework}

The framework introduced by Antunovi\'{c}, Mossel and R\'{a}cz in \cite{AMR} considers a standard preferential attachment graph where the new vertex connects to $m$ existing vertices as originally suggested by Barab\'{a}si and Albert in \cite{scalefree1999}, with the different vertices chosen independently as in the ``independent model'' of \cite{bergerpa}.
That is, we consider a random sequence of graphs $G_0,G_1,\ldots$ starting from some non-empty fixed graph $G_0$.
For each $t>0$ we choose $m$ random vertices from $G_{t-1}$ independently and with replacement, with probabilities proportional to their degrees.
We then add a new vertex connected by $m$ edges to the chosen vertices (allowing multiple edges if a vertex is chosen more than once) to form $G_t$.

Each vertex is one of $N$ types $\{1,\ldots,N\}$, often referred to in \cite{AMR} as colours; once it has been assigned, the type of a vertex is fixed for all time.
When a new vertex joins the graph, it takes a type based on the types of its neighbours in the following general way, where the notation follows section 3 of \cite{AMR}.  The types of the $m$ neighbours induce a vector $\mathbf{u}\in \nn_0^N$ whose elements sum to $m$ and whose $i$th element is the number of neighbours of type $i$.  For each possible $\mathbf{u}$, we will have a probability distribution on $\{1,\ldots,N\}$, which we will describe by a vector $p_{\mathbf{u}} \in \Delta^{N-1}$, whose $i$th element gives the probability that the new vertex is of type $i$ given that $\mathbf{u}$ is the vector giving the numbers of its neighbours of each type.  Here $\Delta^{N-1}$ denotes the $(N-1)$-dimensional simplex.

The special case where $p_{\mathbf{u}}=\mathbf{u}/m$ is referred to in \cite{AMR} as the linear model.  For other cases, they show (Lemma 3.4) that the sequence of vectors $\mathbf{x}_n$ giving the normalised total degrees of the vertices of each type is a stochastic approximation process driven by a vector field $P$ which depends on the $p_{\mathbf{u}}$, that is we have \begin{equation}\label{stochapprox}\mathbf{x}_{n+1}-\mathbf{x}_{n}=\frac{1}{n}\left(P(\mathbf{x}_{n})+\boldsymbol{\xi}_{n+1}+\mathbf{R}_{n}\right),\end{equation} where the $\boldsymbol{\xi}_i$ satisfy the martingale difference condition $\ee(\boldsymbol{\xi}_{n+1}\mid\mathcal{F}_n)=\boldsymbol{0}$, with $\mathcal{F}_n$ being $\sigma(\mathbf{x}_0,\mathbf{x}_1,\ldots,\mathbf{x}_n)$, and the $\mathbf{R}_{i}$ are remainder terms satisfying $\sum_{i=1}^{\infty} \mathbf{R}_i/n <\infty$ almost surely.  As a result, understanding the vector field $P$ is an important step towards understanding the behaviour of the stochastic process and applying the general results on stochastic approximation processes found in, for example, Bena\"{\i}m \cite{benaim} and Pemantle \cite{pemantlesurvey}.

All of the results in \cite{AMR} hold provided that both types are represented in the starting graph, though their model does not in general require this.
In this paper, the models we consider all have the property that a new vertex can only take types which are represented in its neighbourhood, and so we will assume
throughout that all types are represented in the starting graph, since otherwise the missing types never appear and the model reduces to a simpler case.

\section{Rock-paper-scissors model}

We introduce a model in the framework of \cite{AMR} where the type assignment mechanism is based on the game of rock-paper-scissors.  Cyclic dominance systems of this basic form have been shown to naturally occur in a variety of organisms and ecosystems, ranging from colour morphisms of the side-blotched lizard \cite{side-blotched} to strains of \textit{Escherichia coli} \cite{bacteria}. Such patterns of dominance can explain biodiversity. Whereas simpler transitive relations necessarily have a single fittest phenotype which, in the absence of other factors, should eventually dominate, cyclic dominance allows for situations where no phenotype has an evolutionary advantage over all others and thus multiple phenotypes may persist.

Itoh \cite{itoh} investigated a simple Moran process based on the rock-paper-scissors game.
A population of fixed size consists of rock-type, scissors-type and paper-type individuals. At each time step two individuals meet and play rock-paper-scissors; the loser is removed and replaced with a clone of the winner. The population is assumed to be well-mixed, so meetings occur uniformly at random; in such a system one type must eventually take over the whole population.
Similar processes have also been studied in a more structured environment, such as sessile individuals interacting on the $2$-dimensional lattice (see e.g.\ \cite{RPS-lattice}). Here the limited range of interactions allows co-existence of types \cite{bacteria}. While a lattice model may closely approximate the interactions between bacteria growing \textit{in vitro}, neither a lattice nor a well-mixed model is a good representation of the heterogeneous social interaction networks that arise among more complex organisms; here a model incorporating preferential attachment is more realistic.

Our model has $N=3$ types and $m=2$, the types $1$, $2$ and $3$ corresponding to ``rock'', ``paper'' and ``scissors'' respectively.  If the two sampled vertices are of the same type, the new vertex takes that type, whereas if they are of different types they play a game of rock-paper-scissors, playing their type, and the new vertex takes the type of the winner.  In the notation above, we have
\begin{gather*}p_{(1,1,0)}=(0,1,0),\quad p_{(0,1,1)}=(0,0,1),\quad p_{(1,0,1)}=(1,0,0),\\
p_{(2,0,0)}=(1,0,0),\quad p_{(0,2,0)}=(0,1,0),\quad p_{(0,0,2)}=(0,0,1),\end{gather*}
and the vector field $P$, defined by (3.1) of \cite{AMR}, on the triangle $\Delta^2$ is given by the components
\begin{align*}
P_1(x,y,z) &= \frac{x}{2}(z-y) \\ P_2(x,y,z) &= \frac{y}{2}(x-z) \\ P_3(x,y,z) &= \frac{z}{2}(y-x).\end{align*}

\subsection{Limiting behaviour of the model}
Let $X_n$, $Y_n$ and $Z_n$ denote the total degrees of vertices of types $1$, $2$ and $3$ respectively, normalised to sum to $1$. By Lemma 3.4 of \cite{AMR}, $(X_n,Y_n,Z_n)$ follows a stochastic approximation process \eqref{stochapprox} on the triangle $\Delta^2$ driven by the vector field $P$ with the noise terms $\boldsymbol{\xi}_i$ bounded.

\begin{lemma}\label{ODEcirc}The product $xyz$ is constant on trajectories of $P$.\end{lemma}
\begin{proof}We have \[\frac{d}{dt}(xyz)=xyP_3(x,y,z)+xzP_2(x,y,z)+yzP_1(x,y,z)=0.\qedhere\]\end{proof}

The vector field has four stationary points: the corners of the simplex, which are saddle points, and $(1/3,1/3,1/3)$, which has eigenvalues
$\pm\frac{\mathrm{i}}{\sqrt{27}}$, making it an elliptic fixed point.  Together with Lemma \ref{ODEcirc} we can see that trajectories of $P$
circle the centre of the simplex on loops of constant $xyz$; some of these are shown in Figure \ref{trajectories}.

\begin{figure}
\centering\includegraphics[width=.5\textwidth]{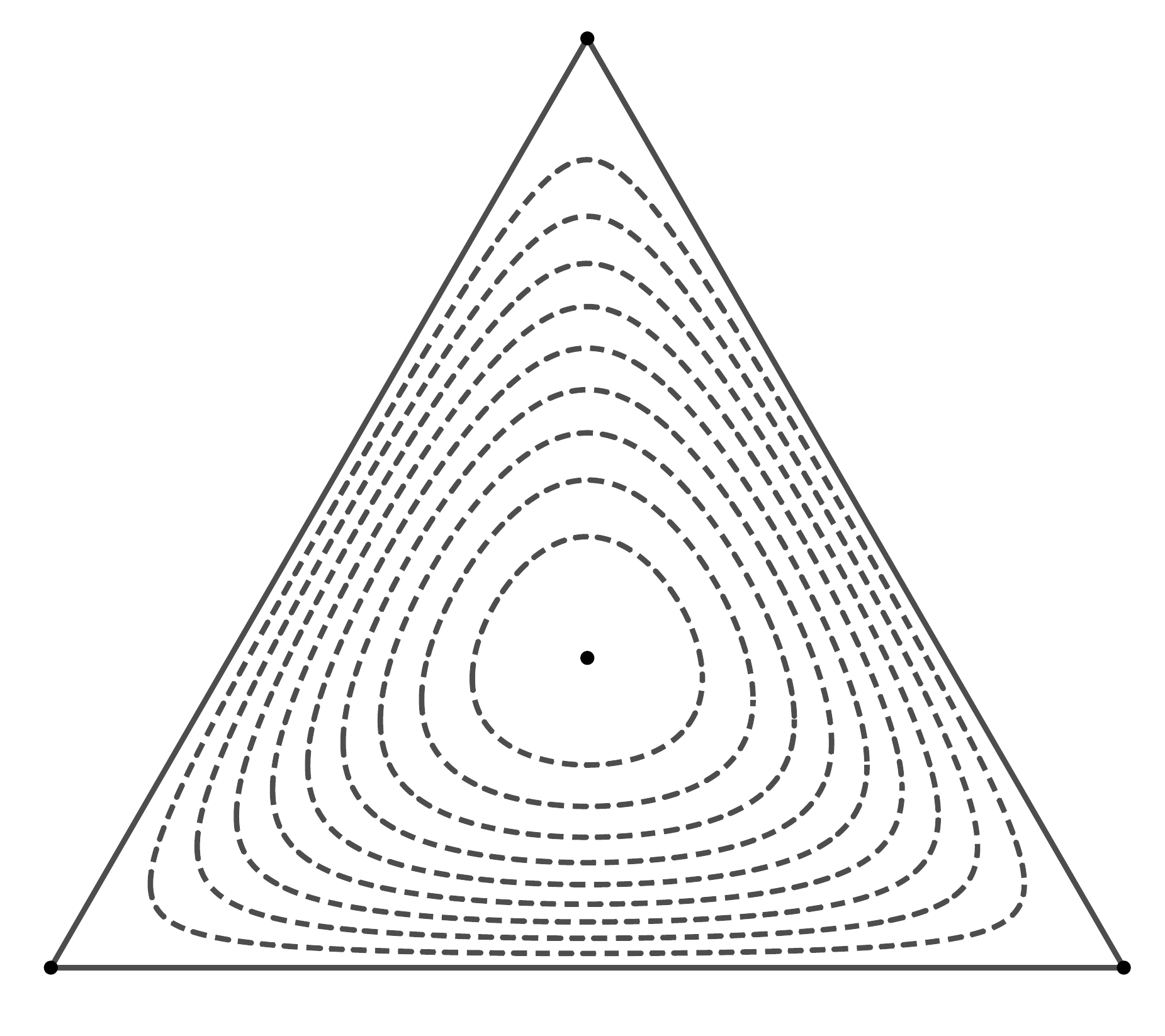}
\caption{Trajectories on which $27xyz$ is constant (ranging from $0.1$ to $0.9$).}\label{trajectories}
\end{figure}

Let $M_n=X_nY_nZ_n$. Our main result is the following.
\begin{theorem}\label{summary}The process $(M_n)_{n\in\nn}$ almost surely converges to a limit $M\in(0,1/27)$, and the distribution of $M$ has full support on $(0,1/27)$.  Furthermore, almost surely $(X_n,Y_n,Z_n)$ fails to converge; rather its limit set is the set $\{(x,y,z)\in\Delta^2:xyz=M\}$.\end{theorem}
\begin{remark}
The failure to converge means that this model provides a counterexample to Conjecture 3.2 of \cite{AMR}.\end{remark}
Theorem \ref{summary} follows from the following three propositions, together with standard results on stochastic approximation processes.
\begin{proposition}\label{main}
The process $(M_n)_{n\in\nn}$ almost surely converges to a limit $M\in[0,1/27]$, and the distribution of $M$ has full support on $(0,1/27)$.
\end{proposition}
\begin{proposition}\label{not-middle}
Almost surely $M=\lim_{n\to\infty}M_n<1/27$.	
\end{proposition}
\begin{proposition}\label{not-edge}
Almost surely $M=\lim_{n\to\infty}M_n>0$.	
\end{proposition}
\begin{proof}[Proof of Proposition \ref{main}]
Let $\gamma_n=4n+2e_0$, that is the total degree in $G_n$; here $e_0$ is the number of edges in the initial graph $G_0$.	Then, if the two sampled vertices are both ``rock'', which has probability $X_n^2$, we have that
\[M_{n+1}\gamma_{n+1}^3=(X_n\gamma_n+4)Y_nZ_n\gamma_n^2=M_n\gamma_n^3+4Y_nZ_n\gamma_n^2,\]
while if one sampled vertex is ``rock'' and the other is ``paper'', which has probability $2X_nY_n$, we have that
\[M_{n+1}\gamma_{n+1}^3=(X_n\gamma_n+1)(Y_n\gamma_n+3)Z_n\gamma_n=M_n\gamma_n^3+3X_nZ_n\gamma_n^2+Y_nZ_n\gamma_n^2+3Z_n\gamma_n,\]
with analogous expressions for the other possibilities.  Hence \begin{align}
\ee(M_{n+1}\mid\salj_n) &= \gamma_{n+1}^{-3}(M_n\gamma_n^3+(4+6+2)M_n\gamma_n^2+(6+6+6)M_n\gamma_n)\nonumber\\
&= M_n(\gamma_n+4)^{-3}(\gamma_n^3+12\gamma_n^2+18\gamma_n)\nonumber\\
&= M_n\left(1-\frac{30}{\gamma_{n+1}^2}+\frac{56}{\gamma_{n+1}^3}\right),\label{mg}\end{align}
showing that $(M_n)_{n\in\nn}$ is a supermartingale.  It takes values in $[0,1/27]$.

If we let $R_{n+1}=M_n\left(\frac{30}{\gamma_{n+1}^2}-\frac{56}{\gamma_{n+1}^3}\right)$ and $\tilde{M}_n=M_n+\sum_{k=1}^n R_k$ then $(\tilde{M})_{n\in\nn}$ is a martingale.  The difference $\tilde{M}_n-M_n$ is bounded, so $\tilde{M}_n\to \tilde{M}$ almost surely, where $\tilde{M}$ is a random limit.

There exist positive constants $c_1,c_2$ such that $-\frac{c_1}{\gamma_n}\leq M_{n+1}-M_n\leq \frac{c_2}{\gamma_n}$.  Hence there exists $c$ such that $\Var(\tilde{M}_{n+1}-\tilde{M}_n\mid\salj_n)\leq \frac{c}{\gamma_n^2}$ and thus $\Var(\tilde{M}\mid\salj_n)\to 0$ as $n\to\infty$.

Given an interval $(r,r+\epsilon)\in(0,1/27)$, for $n$ large enough there will be positive probability of $M_n\in (r+\epsilon/3,r+2\epsilon/3)$.  That $\Var(\tilde{M}\mid\salj_n)\to 0$ and that $\sum_{k=n+1}^{\infty}R_k \to 0$ as $n\to\infty$ ensures that if $n$ is large enough there is then positive probability of $M\in (r,r+\epsilon)$.
\end{proof}
In order to prove Proposition \ref{not-middle}, we will need better control on the variation of $M_{n+1}$ if $M_n$ is close to $1/27$.
\begin{lemma}\label{middle-variance}If $M_n>\frac{1}{27}-\frac{c}{\gamma_n}$ then $\babs{M_{n+1}-\ee(M_{n+1}\mid\salj_n)}<\frac{C}{\gamma_n^{3/2}}$ for some $C$ which depends only on $c$ and for sufficiently large $n$.\end{lemma}
\begin{proof}
Note that if $\abs{X_n-\frac13}\geq \frac{c'}{\sqrt{\gamma_n}}$ then $M_n\leq\frac{1}{27}-\frac{c'^2}{4\gamma_n}+\frac{c'^3}{4\gamma_n^{3/2}}$. Consequently, for a suitable
choice of $c'$ and sufficiently large $n$ we have $\abs{X_n-\frac13},\abs{Y_n-\frac13},\abs{Z_n-\frac13}<\frac{c'}{\sqrt{\gamma_n}}$. In turn this means that
\begin{align*}\frac{1}{X_n}&<\frac{3}{1-\frac{3c'}{\sqrt{\gamma_n}}}\\
&=3+\frac{9c'}{\sqrt{\gamma_n}}+\frac{27c'^2}{\gamma_n-3c'\sqrt{\gamma_n}}\\
&<3+\frac{c''}{\sqrt{\gamma_n}}\end{align*}
for some $c''$ and $n$ sufficiently large. Similarly we have $\frac{1}{X_n}>3-\frac{c''}{\sqrt{\gamma_n}}$, and the same bounds apply to $Y_n,Z_n$.

Write $\mu_{n+1}$ for $\ee(M_{n+1}\gamma_{n+1}^3\mid\salj_n)$; by \eqref{mg}, since $\gamma_{n+1}=\gamma_n+4$, we have
\begin{align*}\mu_{n+1}&=M_n\gamma_{n+1}^3-30M_n\gamma_{n+1}+56M_n\\
&=M_n\gamma_n^3+12M_n\gamma_n^2+18M_n\gamma_n.\end{align*}

With probability $X_n^2$ we have $M_{n+1}\gamma_{n+1}^3=M_n\gamma_n^3+4Y_nZ_n\gamma_n^2$, giving
\begin{align*}\abs{M_{n+1}\gamma_{n+1}^3-\mu_{n+1}}&\leq\Bigl|\frac{4}{X_n}-12\Bigr|M_n\gamma_n^2+18M_n\gamma_n\\
&<\frac{4c''}{27}\gamma_n^{3/2}+\frac{2}{3}\gamma_n<C\gamma_n^{3/2}
\end{align*}
for some $C$ and sufficiently large $n$. With probability $2X_nY_n$ we have $M_{n+1}\gamma_{n+1}^3=M_n\gamma_n^3+3X_nZ_n\gamma_n^2+Y_nZ_n\gamma_n^2+3Z_n\gamma_n$, giving
\begin{align*}\abs{M_{n+1}\gamma_{n+1}^3-\mu_{n+1}}&\leq\Bigl|\frac{3}{Y_n}-9+\frac{1}{X_n}-3\Bigr|M_n\gamma_n^2+\abs{3Z_n-18M_n}\gamma_n\\
&<\frac{4c''}{27}\gamma_n^{3/2}+3\gamma_n+\frac{2}{3}\gamma_n<C\gamma_n^{3/2},
\end{align*}
and similar bounds apply in other cases. Thus we have
\[\babs{M_{n+1}-\ee(M_{n+1}\mid\salj_n)}<\frac{C\gamma_n^{3/2}}{\gamma_{n+1}^3}<\frac{C}{\gamma_n^{3/2}}.\qedhere\]
\end{proof}
We are now ready to show that almost surely $M_n$ does not tend to $1/27$.
\begin{proof}[Proof of Proposition \ref{not-middle}]
Suppose for the sake of contradiction that $\prb{M=\frac1{27}}>0$. Then for $n_0$ sufficiently large there will be an event $\mathcal A\in\salj_{n_0}$ such that
$\prb{M=\frac1{27}\mid\mathcal A}\geq 1-\eps$.
Fix $c_1>c_2>0$ to be chosen later, and let $\mathcal B$ be the event that for some $n\in[n_0,2n_0]$ we have $M_{n}\in\bbrac{\frac1{27}-\frac{c_1}{n},\frac1{27}-\frac{c_2}{n}}$.
We claim that, for suitable $c_1,c_2$ which do not depend on $n_0$, we have $\pr(\mathcal B\mid\salj_{n_0})$ is bounded away from $0$ for $n_0$ sufficiently large.
To see this, stop the process if $M_n<\frac1{27}-\frac{c_2}{n_0}$ or if $n=2n_0$; write $\tau$ for the stopping time. By choice of $\tau$, it follows from \eqref{mg} that
$\hier[M]{\tau}{n+1}+\sum_{k=n_0}^{\min(n,\tau)}\frac 1{\gamma_k^2}$ is a supermartingale, since $30\hier[M]{\tau}n\geq30\bbrac{\frac1{27}-\frac{c_1}{n_0}}>1$.

If $\mathcal B$ fails, we must have $\tau=2n_0$ and $M_{2n_0}>\frac1{27}-\frac{c_2}{2n_0}$, \ie
\[\hier[M]{\tau}{2n_0}+\sum_{k=n_0}^{2n_0-1}\frac 1{\gamma_k^2}>\frac1{27}+\frac a{n_0}\]
for some constant $a$, which is positive for suitable choice of $c_2$. Applying Azuma--Hoeffding, using Lemma \ref{middle-variance},
this occurs with probability bounded away from $1$.

Suppose $\mathcal B$ occurs, with $\tau=n_1$. Let $\mathcal C$ be the event that $M_n\in\bbrac{\frac1{27}-\frac{2c_1}{n_1},\frac1{27}-\frac{c_2}{2n_1}}$ for every $n\geq n_1$.
We claim that $\pr(\mathcal C\mid\salj_{n_1})$ is bounded away from $0$. The proof is similar: fix $n_2>n_1$ and stop the process if it leaves the interval or if $n=n_2$,
with stopping time $\tau'$.
If $\mathcal C$ fails to hold before $n_2$ then we have, evaluated at $n=n_2$, either
\begin{equation}\hier[M]{\tau'}{n}>M_{n_1}-\frac{c_2}{2n_1}\label{toohigh}\end{equation} or
\begin{equation}\hier[M]{\tau'}{n}-\sum_{k=n_1}^{n-1}\frac{2}{\gamma_k^2}<M_{n_1}-\frac{c_1}{n_1}\label{toolow}.\end{equation}
Since the left-hand sides of \eqref{toohigh} and \eqref{toolow} are respectively a supermartingale and submartingale, with variations bounded by Lemma \ref{middle-variance},
again by Azuma--Hoeffding this has probability bounded away from $1$, where the bound is independent of $n_1$ and $n_2$.
\end{proof}
Finally, we show that almost surely the limit $M$ is positive.
\begin{proof}[Proof of Proposition \ref{not-edge}]
First we claim that almost surely $M_n=\Omega(\gamma_n^{-1})$. Indeed, in a standard preferential attachment process the degree of a fixed vertex $v_i$ satisfies $d_n(v_i)=(1+o(1))\xi_i\sqrt{\gamma_n}$, where $\xi_i$ is a random variable which is almost surely positive: see Theorem 8.2, Lemma 8.17 and Exercise 8.13 of \cite{vdH}.
Thus the contribution of the starting vertices alone ensures that $\min(X_n,Y_n,Z_n)=\Omega(\gamma_n^{-1/2})$ and so $X_nY_nZ_n=\Omega(\gamma_n^{-1})$.

As in the proof of Lemma \ref{middle-variance}, with probability $X_n^2$ we have
\[M_{n+1}\gamma_{n+1}^3-\mu_{n+1}=(4Y_nZ_n-12M_n)\gamma_n^2-18M_n\gamma_n;\]
note that
\[X_n^2\bigl((4Y_nZ_n-12M_n)\gamma_n^2-18M_n\gamma_n\bigr)^2=O(M_n^2\gamma_n^4).\]
With probability $2X_nY_n$ we have
\[M_{n+1}\gamma_{n+1}^3-\mu_{n+1}=(3X_nZ_n+Y_nZ_n-12M_n)\gamma_n^2+(3Z_n-18M_n)\gamma_n,\]
and
\[2X_nY_n\bigl((3X_nZ_n+Y_nZ_n-12M_n)\gamma_n^2+(3Z_n-18M_n)\gamma_n\bigr)^2=O(M_n\gamma_n^4).\]
Similar expressions hold for the other possibilities, giving $\Var(M_{n+1}\gamma^3_{n+1}\mid\salj_n)=O(M_n\gamma_n^4)$,
\ie $\Var(M_{n+1}\mid\salj_n)=O(M_n\gamma_n^{-2})$.

Suppose $M_{n'}<2M_n$ for all $n'\geq n$. Then we have $\Var(M_{n'+1}\mid\salj_{n'})=O(M_n\gamma_n^{-2})$
for each $n'\geq n$, giving $\Var(M\mid\salj_n)=O(M_n\gamma_n^{-1})=O(M_n^2)$. Thus
there is a probability bounded away from $0$ as $n\to\infty$ that $M$ is in the
interval $(M_n/2, 3M_n/2)$ conditional on $\salj_n$, but if $M=0$ has positive
probability then for any $\eps>0$ and $n$ sufficiently large there is an event
$\mathcal A\in\salj_n$ with $P(M=0\mid\mathcal A)>1-\eps$, giving a contradiction.\end{proof}

We can now complete the proof of our main result.
\begin{proof}[Proof of Theorem \ref{summary}]
Propositions \ref{main}, \ref{not-middle} and \ref{not-edge} show that the limit set $L(X,Y,Z)$ is, almost surely, contained within $\{(x,y,z)\in\Delta^2: xyz=M\}$, where $M\in(0,1/27)$ is the random variable defined in Proposition \ref{main}.  By Theorem 5.7 of Bena\"{\i}m \cite{benaim}, $L(X,Y,Z)$ is almost surely a chain transitive set for the flow; here a chain transitive set for the flow is an invariant set $M$ for the flow such that for every pair of points $a$ and $b$ in $M$ and for any $\delta>0$ and $T>0$ there is a $(\delta,T)$-pseudo-orbit from $a$ to $b$, meaning a finite sequence of partial trajectories, with the first starting at $a$ and the last starting at $b$, the duration of each trajectory at least $T$, and the finishing point of one trajectory and the starting point of the next at most $\delta$ apart.
 For $M\in(0,1/27)$ the only invariant set for the flow, and hence the only chain transitive set, which is a subset of $\{(x,y,z)\in\Delta^2: xyz=M\}$ is the whole set.
\end{proof}

\begin{figure}
\centering\includegraphics[width=.9\textwidth]{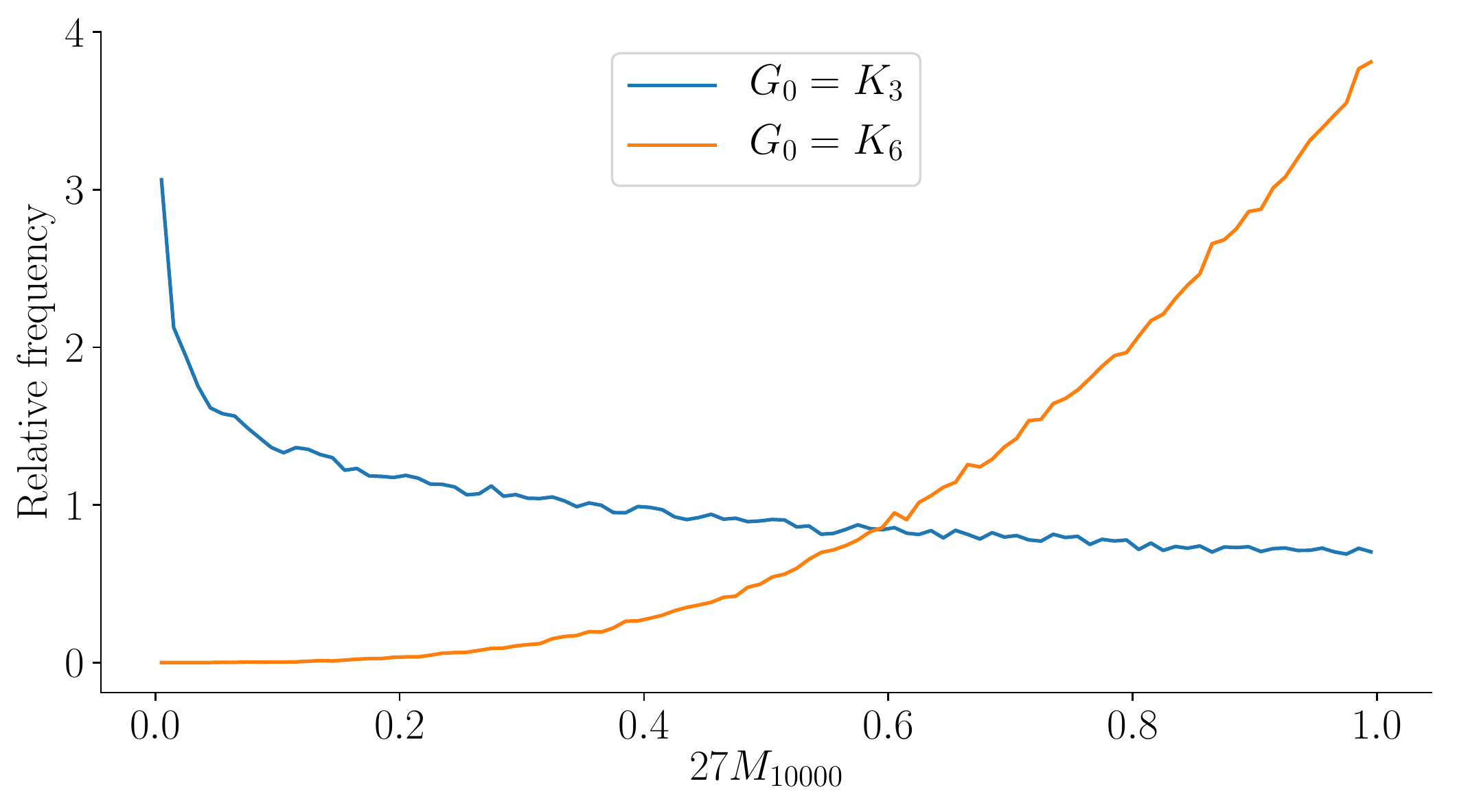}
\caption{Distributions of the value of $27M_{10000}$ from two different starting graphs.}\label{fig:dist}
\end{figure}

The distribution of $M$ will naturally depend critically on the starting graph $G_0$. Figure \ref{fig:dist} shows approximate distributions for two particular choices of $G_0$, being the complete graphs on $3$ and $6$ vertices respectively, each with equal numbers of rock, paper and scissors vertices. These distributions were taken from simulations to time $10000$.

\begin{figure}
\centering\includegraphics[width=.9\textwidth]{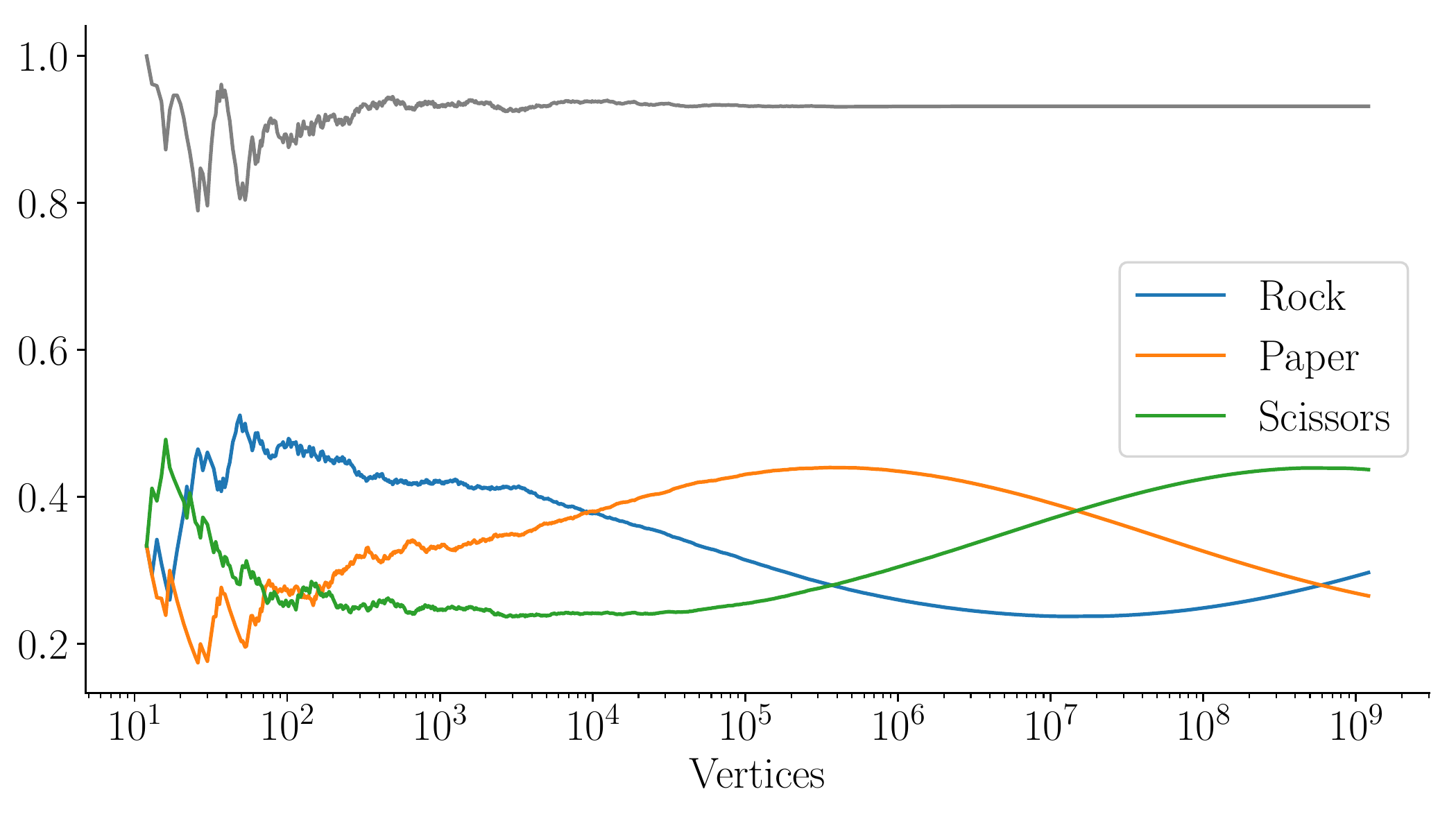}
\caption{Evolution of $X_n,Y_n,Z_n$ from a simulation, together with $27M_n$ (grey curve).}\label{circling}
\end{figure}

\subsection{Rate of circling}

In this section we show that circling around the limiting cycle occurs on a logarithmic scale as $n\to\infty$, at a rate which depends only on the limit parameter $M$; this is consistent with the behaviour seen in Figure \ref{circling}.

\begin{theorem}For $n_0$ sufficiently large depending on $M=\lim_{n\to\infty}M_n$, with high probability the process completes a circuit
approximating the trajectory $M_n=M$ at time $(A+o(1))n_0$, where $A>1$ is a parameter which depends only on $M$.\end{theorem}
\begin{proof}For $(x,y,z)\in\Delta^2$, write $f(x,y,z)$ for $\bigl\Vert\frac{x(z-y)}2,\frac{y(x-z)}2,\frac{z(y-x)}2\bigr\Vert$.
For any $\delta>0$, $f(x,y,z)$ is bounded away from $0$ whenever $\min(x,y,z)\in(\delta,1/3-2\delta)$, since assuming without loss of
generality that $x$ is the median value we have $\abs{x(z-y)}>\delta/3$, and trivially $f$ is also bounded above. Similarly we may bound
all partial derivatives of $f(x,y,z)$ away from $0$ when $\min(x,y,z)\in(\delta,1/3-2\delta)$.

Let $\mathcal C_M$ be the curve $\{(x,y,z)\in\Delta^2:xyz=M\}$, and let $L_M$ be its length (in the Euclidean metric).
Fix $\delta>0$ such that $M\in(\delta,1/27-\delta)$. Let $\eps\in(0,\delta^2)$ be arbitrary, and suppose $n_0$ is sufficiently large
that $\abs{M_n-M}<\eps^2$ for all $n>n_0$ with high probability. Note that, conditioned on this event, we must have
$\min(X_n,Y_n,Z_n)\in(\delta,1/3-2\delta)$ for all $n>n_0$ and so $f(X_n,Y_n,Z_n)$ is bounded away from $0$.

Write $n_{i+1}=\floor{(1+\eps)n_i}$ for $i\geq 0$ and consider the process $(X_n,Y_n,Z_n)$ for $n\in[n_0,n_1]$. Think of this as an urn process,
where balls represent edge-ends; for each vertex we draw two balls from the urn, replace them and add four new balls depending on the draw.
For the moment, only reveal the information of whether each ball drawn was in the urn at time $n_0$ or not; call a vertex ``typical'' if
both balls drawn for that vertex were in at time $n_0$. The number of atypical vertices is dominated by a binomial $(\floor{\eps n_0}, 2\eps)$
random variable, so we have at least $(\eps-3\eps^2)n_0$ typical vertices with high probability. Now the type of new balls added for each
typical vertex are independent and identically distributed, contributing on average $4\eps n_0X_{n_0}\bbrac{1+\frac{Z_{n_0}-Y_{n_0}}2}$ rock,
$4\eps n_0Y_{n_0}\bbrac{1+\frac{X_{n_0}-Z_{n_0}}2}$ paper and $4\eps n_0Z_{n_0}\bbrac{1+\frac{Y_{n_0}-X_{n_0}}2}$ scissors to the urn,
so the variance of numbers of each type contributed by typical vertices is $O(\eps n)=o(\eps^2n^2)$. Consequently with high probability
at time $n_1$ the number of balls of type rock is at least $4n_0X_{n_0}+4\eps nX_{n_0}\bbrac{1+\frac{Z_{n_0}-Y_{n_0}}2}-4\eps^2n$ and at most
$4n_0X_{n_0}+4\eps nX_{n_0}\bbrac{1+\frac{Z_{n_0}-Y_{n_0}}2}+4\eps^2n_0$, and similarly for paper and scissors.

It follows that with high probability the distance from $(X_{n_0},Y_{n_0},Z_{n_0})$ to $(X_{n_1},Y_{n_1},Z_{n_1})$ is within $12\eps^2$ of
$\frac{\eps}{2(1+\eps)}f(X_{n_0},Y_{n_0},Z_{n_0})$, and similarly with high probability the distance from $(X_{n_i},Y_{n_i},Z_{n_i})$ to $(X_{n_{i+1}},Y_{n_{i+1}},Z_{n_{i+1}})$
is within $12\eps^2$ of $\frac{\eps}{2(1+\eps)}f(X_{n_i},Y_{n_i},Z_{n_i})$ for each $i$.
Note that, since partial derivatives of $f$ are
bounded away from $0$, $f(X_{n_i},Y_{n_i},Z_{n_i})^{-1}$ is within $O(\eps^{2})$ of $f(x_i,y_i,z_i)^{-1}$, where we define $(x_i,y_i,z_i)$ to be the closest
point to $(X_{n_i},Y_{n_i},Z_{n_i})$ on $\mathcal C_M$.

There exist values $b<B$ such that after at least $b/\eps$ and at most $B/\eps$ steps of this form the process has completed a circuit.
The time at which this occurs is therefore in the interval $[(1+\eps)^{b/\eps}n_0,(1+\eps)^{B/\eps}n_0]$, \ie in $[\mathrm{e}^bn_0,\mathrm{e}^Bn_0]$.

Letting $\eps\to 0$ gives the required result with $A=\exp\bbrac{2L_M\int_{t\in\mathcal C_M}f(t)^{-1}\,\mathrm{d}t}$.\end{proof}

\subsection{Affine preferential attachment}

A natural extension of the model of \cite{AMR}, considered briefly in that paper, is where we have affine preferential attachment so that a vertex $v$ is chosen for attachment with probability proportional to $d(v)+\alpha$ for some $\alpha>-2$.  Affine preferential attachment was introduced by Dorogovtsev, Mendes and Samukhin in \cite{dorog}.  It turns out that  the behaviour of the rock-paper-scissors model is similar in this modified setting.  Let $X_n,Y_n,Z_n$ be the probabilities of selecting rock, paper and scissors vertices respectively by a single preferential choice at time $n$, let $M_n=X_nY_nZ_n$, and write $\gamma_n=\sum_v(d_n(v)+\alpha)$. Now we have $\gamma_{n+1}^3M_{n+1}=(\gamma_nX_n+4+\alpha)\gamma_nY_n\gamma_nZ_n$ with probability $X_n^2$, $\gamma_{n+1}^3M_{n+1}=(\gamma_nX_n+1)(\gamma_nY_n+3+\alpha)\gamma_nZ_n$ with probability $2X_nY_n$, and so on, giving
\begin{align*}
\ee(M_{n+1}\mid\salj_n) &= \gamma_{n+1}^{-3}(M_n\gamma_n^3+(12+3\alpha)M_n\gamma_n^2+(18+6\alpha)M_n\gamma_n)\\
&= M_n(\gamma_n+4+\alpha)^{-3}(\gamma_n^3+(12+3\alpha)\gamma_n^2+(18+6\alpha)\gamma_n)\\
&= M_n\left(1-\frac{(30+18\alpha+3\alpha^2)\gamma_n}{\gamma_{n+1}^3}-\frac{(4+\alpha)^3}{\gamma_{n+1}^3}\right).
\end{align*}
As in the proof of Proposition \ref{main}, noting that $30+18\alpha+3\alpha^2\geq 3$, we have that $(M_n)_{n\in\nn}$ is a supermartingale with the appropriate variance properties, meaning that Propositions \ref{main} and \ref{not-middle} apply as in the standard model. For $\alpha>0$, however, the proof of Proposition \ref{not-edge} does not translate to this setting, since the degree of a given vertex grows as $\gamma_n^{1/(2+\alpha/2)}$.

If $M=0$ then $\{(x,y,z)\in\Delta^2: xyz=0\}$ is a chain transitive set, but the stationary points at the corners of the triangle are also chain transitive sets.  However, we can prove that the corners are limits with probability zero.  Without loss of generality, assume $(X_n,Y_n,Z_n)\to (1,0,0)$.  Then, for $n$ sufficiently large $X_n>1/2$, meaning that conditional on $\salj_n$ the probability that vertex $n+1$ is of type $2$ (paper), $Y_n^2+2X_nY_n>Y_n$.  Consequently we can bound $Y_n$ below by the proportion of black balls in a coupled standard P\'{o}lya urn process, showing that $\pr(Y_n\to 0)=0$ on the event $X_n>1/2$ for $n$ large enough.  Hence $\pr(L(X,Y,Z)=(1,0,0))=0$. Thus we have the following slight weakening of Theorem \ref{summary} for this setting.
\begin{theorem}For affine preferential attachment, the process $(M_n)_{n\in\nn}$ almost surely converges to a limit $M\in[0,1/27)$, and the distribution of $M$ has full support on $(0,1/27)$.  Furthermore, almost surely $(X_n,Y_n,Z_n)$ fails to converge; rather its limit set is $\{(x,y,z)\in\Delta^2:xyz=M\}$.\end{theorem}

\section{Pick random visible type}

We now consider another natural, simple rule for choosing types; instead of copying the type of a random neighbour, as in the linear model, we choose uniformly at random between types present in the neighbourhood. This method gives common types slightly less advantage than the linear model, and instead of almost sure convergence to a random limit, here we obtain almost sure convergence to a deterministic limit.

\begin{theorem}Suppose we have $N\geq 2$ types and each new vertex chooses $m\geq 3$ neighbours, and adopts a type chosen uniformly at random from those present among its neighbours. Then the proportion of each type converges almost surely to $1/N$.\end{theorem}
\begin{remark}If $m=2$ then this model reduces to the linear model of \cite{AMR}.\end{remark}
\begin{proof}Write $\hier in$ for the proportion of edge-ends from vertices of type $i$. It is sufficient to show that $\liminf_{n\to\infty}\hier Nn\geq 1/N$ almost surely, since by symmetry of the model the same will apply to all other types, implying that $\hier in\to 1/N$; convergence of the proportions of vertices follows by considering the vertices added once $\abs{\hier in-1/N}<\eps$ for each $i$.

We couple the process with a two-type process as follows. Treat all types other than type $N$ as indistinguishable, forming a single supertype $*$, and let $Y_n$ be the proportion of edge-ends of type $N$ at time $n$. Join each new vertex to $m$ vertices as before.  If each of the chosen vertices has the same type, assign that type to the new vertex. Otherwise, if $k$ vertices are chosen from type $*$ and $m-k$ from type $N$ with $0<k<m$, sample $k$ independent variables from the uniform distribution on $\{1,\ldots,n-1\}$ and let $Z(k)$ be the number of different values seen; now assign type $N$ to the new vertex with probability $\frac1{Z(k)+1}$.

By Lemma \ref{other-types-equal} below, for every $k$ and $j$ we have
\[\pr(Z(k)\geq j)\geq\pr(A_{n+1}\geq j\mid B_{n+1}=k,\salj_n),\]
where $A_{n+1}$ is the number of types other than $N$ among neighbours of $v_{n+1}$ in the original process,
and $B_{n+1}$ is the number of neighbours not of type $N$; it follows that
\[\ee\Bigl(\frac1{Z(k)+1}\Bigr)\leq\ee\Bigl(\frac1{A_{n+1}+1}\Bigm| B_{n+1}=k,\salj_n\Bigr).\]
Provided $\hier Nn\geq Y_n$, we have
\begin{align*}\pr(v_{n+1}\text{ has type }N\mid\salj_n)&=\sum_{k=0}^{m-1}\pr(B_{n+1}=k)\ee\Bigl(\frac1{A_{n+1}+1}\Bigm|B_{n+1}=k,\salj_n\Bigr)\\
&\leq\sum_{k=0}^{m-1}\binom mk(1-Y_n)^k(Y_n)^{m-k}\ee\Bigl(\frac1{Z(k)+1}\Bigr),\end{align*}
and so it is possible to couple the two processes such that $\hier Nn\geq Y_n$.
Writing
\[f(y)=\sum_{k=0}^{m-1}\binom mk(1-y)^ky^{m-k}\ee\Bigl(\frac1{Z(k)+1}\Bigr)-y,\]
we have
\[Y_{n+1}-Y_n=\frac{f(Y_n)+\xi_{n+1}}{\gamma_{n+1}},\]
where $\gamma_n$ is the number of edge-ends at time $n$ and $\xi_n$ is a random variable satisfying $\abs{\xi_n}<m+1$ and $\ee(\xi_n\mid\salj_n)=0$. It is straightforward to check that this is a one-dimensional stochastic approximation process satisfying the conditions of Pemantle \cite{pemantlesurvey}, Section 2.4, and hence Corollary 2.7 of \cite{pemantlesurvey} implies that $Y_n$ converges to a zero of $f$.

We claim that $f(y)>0$ for $0<y<1/N$. To see this, note that $f(y)$ is the difference in probability of the new vertex selecting type $N$ in the linear model (copying the type of a random neighbour) over this model, assuming that the proportion of type-$N$ edge ends is $y$, and proportions of other types are equal. We condition on the types represented in the neighbourhood; the only cases which contribute are those where $N$ is represented. Given that type $N$ and $k$ specified other types are represented, the expected number of neighbours of type $N$ is at most that of each other type, so is at most $\frac m{k+1}$. Thus the probability of selecting type $N$, given which types are represented, is no greater in the linear model than in this model. This inequality is strict provided $0<k<m-1$ (if $k=m-1$ then necessarily each type is represented exactly once). Since $m\geq 3$ and $N\geq 2$, the inequality is strict in at least one case with positive probability of occurring, and so $f(y)>0$.

Thus $Y_n\to 0$ or $\lim Y_n\geq 1/N$, so it suffices to show that $Y_n\not\to 0$. This follows since if $Y_n\to 0$ then we have $Y_n<1/N$ for $n$ sufficiently large, and since $f(y)\geq0$ if $y\leq1/N$ we can couple to a standard P\'{o}lya urn.
\end{proof}
We conclude by proving the inequality required for the two-type coupling.
\begin{lemma}\label{other-types-equal}Fix $n\geq 1$, $m\geq 0$ and $k\geq 0$, and let $\mathbf p$ be a probability distribution on $[n]$. Then the probability $p_{n,m,k}(\mathbf{p})$ that a sample of $m$ independent variables with distribution $\mathbf p$ includes at least $k$ different elements of $[n]$ is maximised when $\mathbf p=(1/n,\ldots,1/n)$, and moreover when $n,m\geq k\geq 2$ that is the unique maximiser.\end{lemma}
\begin{proof}The statement is trivial unless $n,m\geq k\geq 2$ since if $\min(n,m)<k$ then $p_{n,m,k}(\mathbf p)\equiv 0$ and if $n,m\geq k<2$ then $p_{n,m,k}(\mathbf p)\equiv 1$. When $n,m\geq k\geq 2$, we prove that if $\mathbf p=p_1,\ldots,p_n$ is a non-uniform probability distribution then it does not maximise $p_{n,m,k}(\mathbf p)$ by induction on $n$, with base case $n=2$; in this case we have $p_{n,m,k}(\mathbf p)=1-p_1^m-(1-p_1)^m$ and it is easy to see (e.g.\ by calculus) that this is uniquely maximised when $p_1=1/2$.

Suppose $n>2$ but the result holds for smaller values of $n$; note that any distribution with full support gives $p_{n,m,k}(\mathbf p)>0$, and so we may assume that $p_i\neq 1$ for each $i$. If additionally $\mathbf p$ is non-uniform, there exists some $i$ for which $p_i>0$ and the other probabilities are not all the same; without loss of generality, assume $i=n$. We condition on the number of times $n$ appears in the sample, so that
\[p_{n,m,k}(\mathbf p)=(1-p_n)^mp_{n-1,m,k}(\mathbf q)+\sum_{j=1}^m\binom mjp_n^j(1-p_n)^{m-j}p_{n-1,m-j,k-1}(\mathbf q),\]
where $\mathbf q=\bigl(\frac{p_1}{1-p_n},\ldots,\frac{p_{n-1}}{1-p_n}\bigr)$ is the conditional distribution if $n$ is not selected. Applying the induction hypothesis to $\mathbf q$, equalising $p_1,\ldots,p_{n-1}$ will not decrease any term, and will strictly increase at least one term (the initial term if $n>k$ or the $j=1$ term otherwise), so $\mathbf p$ does not maximise $p_{n,m,k}(\mathbf p)$.\end{proof}

\bibliographystyle{amsplain}
\bibliography{scalefree}

\providecommand{\bysame}{\leavevmode\hbox to3em{\hrulefill}\thinspace}
\providecommand{\MR}{\relax\ifhmode\unskip\space\fi MR }
\providecommand{\MRhref}[2]{%
  \href{http://www.ams.org/mathscinet-getitem?mr=#1}{#2}
}
\providecommand{\href}[2]{#2}
\begin{thebibliography}{10}

\bibitem{AMR}
T.~Antunovi{\'{c}}, E.~Mossel, and M.~R{\'{a}}cz, \emph{Coexistence in
  preferential attachment networks}, Combinatorics, Probability and Computing
  \textbf{25} (2016), 797--822.

\bibitem{scalefree1999}
A.-L. Barab{\'a}si and R.~Albert, \emph{Emergence of scaling in random
  networks}, Science \textbf{286} (1999), 509--512.

\bibitem{benaim}
Michel Bena{\"\i}m, \emph{Dynamics of stochastic approximation algorithms},
  S{\'e}minaire de Probabilit{\'e}s, XXXIII \textbf{1709} (1999), 1--68.

\bibitem{bergerpa}
N.~Berger, C.~Borgs, J.~Chayes, and A.~Saberi, \emph{Asymptotic behaviour and
  distributional limits of preferential attachment graphs}, Annals of
  Probability \textbf{42} (2014), 1--40.

\bibitem{dorog}
S.~N. Dorogovtsev, J.~F.~F. Mendes, and A.~N. Samukhin, \emph{Structure of
  growing networks with preferential linking}, Phys. Rev. Lett. \textbf{85}
  (2000), 4633--4636.

\bibitem{vdH}
R.~van~der Hofstad, \emph{Random graphs and complex networks. {V}ol. 1},
  Cambridge University Press, 2017.

\bibitem{itoh}
Y.~Itoh, \emph{On a ruin problem with interaction}, Ann.\ Instit.\ Statst.\
  Math. \textbf{25} (1973), 635--641.

\bibitem{bacteria}
B.~Kerr, M.~A. Riley, M.~W. Feldman, and B.~J.~M. Bohanna, \emph{Local
  dispersal promotes biodiversity in a real-life game of rock-paper-scissors},
  Nature \textbf{418} (2002), 171--174.

\bibitem{pemantlesurvey}
R.~Pemantle, \emph{A survey of random processes with reinforcement},
  Probability Surveys \textbf{4} (2007), 1--79.

\bibitem{side-blotched}
B.~Sinervo and C.~M. Lively, \emph{The rock-paper-scissors game and the
  evolution of alternative male strategies}, Nature \textbf{380} (1996),
  240--243.

\bibitem{RPS-lattice}
G.~Szab\'o, A.~Szolnoki, and R.~Izs\'ak, \emph{Rock-scissors-paper game on
  regular small-world networks}, J. Physics A \textbf{37} (2004), 2599--2609.

\end{thebibliography}
\end{document}